\documentclass[a4wide,12pt,reqno,final]{amsart}
\usepackage[T1]{fontenc}
\usepackage{amssymb}
\usepackage{amsthm}
\usepackage{amsmath}
\usepackage{xspace}
\usepackage{url}
\usepackage{graphics}
\usepackage{lineno}
\usepackage{tikz}
\usepackage{pstricks,pst-node}
\usepackage[bookmarks=true,bookmarksopen=true,breaklinks=true,colorlinks=true,linkcolor=blue,anchorcolor=blue,citecolor=blue,filecolor=blue,menucolor=blue,pagecolor=blue,urlcolor=blue]{hyperref}

\numberwithin{equation}{section}

\newcommand{\N}{\mathbb{N}}

\newcommand{\R}{\mathbb{R}}

\theoremstyle{plain}% default
\newtheorem{thm}{Theorem}[section]
\newtheorem{lem}[thm]{Lemma}
\newtheorem{prop}[thm]{Proposition}
\newtheorem{cor}{Corollary}
\theoremstyle{definition}

\theoremstyle{remark}
\newtheorem*{rem}{Remark}

\newcommand*\deuxlignes[1]{%
\begin{tabular}{@{}c@{}}
#1%
\end{tabular}%
}
\begin{document}
\title[Extension of continuous functions and w. d. spaces]{Extension of continuous functions and weakly developable spaces}
\author{Boualem ALLECHE}
\address{Laboratoire de M\'ecanique, Physique et
Mod\'elisation Math\'ematique,\\
Universit\'e de  M\'ed\'ea, 26000 M\'ed\'ea, Algeria}
\email{alleche.boualem@univ-medea.dz}
\email{alleche.boualem@gmail.com} \dedicatory{Laboratoire de
M\'ecanique, Physique et Mod\'elisation Math\'ematique.\\
Universit\'e de M\'ed\'ea. Ain Dheb.\\
26000 M\'ed\'ea. Algeria.\\
alleche.boualem@univ-medea.dz, alleche.boualem@gmail.com}
\begin{abstract}
The aim of this paper is to investigate weakly developable spaces.
For a comparison with semi-metrizable spaces, we introduce and
study a class of spaces among those of weakly developable spaces,
semi-metrizable spaces and first countable spaces having a
$G_\delta$-diagonal. Some results are obtained and applications to
the problem of extending continuous functions are discussed.
\end{abstract}
\date{\today}
\maketitle

\section{Introduction}
The problem of extending a continuous function $f:A\longrightarrow
Y$ from a closed subset $A$ of a space $X$ to some neighborhood
$U$ of $A$ in $X$ or to all of $X$ is very often encountered in
many areas of mathematics. Tietze-Urysohn extension theorem solves
positively the problem when $Y=\R$ and $X$ normal (metric and
compact spaces are normal). Several generalizations and other
extension theorems have been also obtained by J. Dugundji, C.H.
Dowker, K. Borsuk, C. R. Borges and many other mathematicians. To
deal with this problem, E. Michael introduced the notion of
selection theory which is today an indispensable tool for a large
number of those working in topology, functional analysis,
set-valued analysis, dynamical systems, control theory,
approximation theory, convex geometry, economic mathematics and so
on. J. Calbrix and the present author introduced the notion of
weakly developable space mainly to generalize a theorem of E.
Michael on double selection. This notion fits very well in the
class of generalized metric spaces.

The purpose of this paper is to investigate weakly developable
spaces by introducing and studying a notion which joins some
properties common to weakly developable spaces and semi-metrizable
spaces. Some results and characterizations are established.
Applications to the problem of extending continuous functions and
set-valued functions as well as to different other problems are
discussed.

The first version of this paper was originally written to be
presented during RAMA'8, 26-29 November, 2012, Algiers, Algeria.
\section{Notations and preliminary results}
Throughout this paper, all spaces are assumed to be
$T_1$-topological spaces. The notions not defined here can be
found in \cite{bu,en,ggr}.

Let $X$ be a topological space. For a collection ${\mathcal A}$ of
subsets of topological space $X$, we write ${\overline{{\mathcal
A}}}=\left\{{\overline A}\mid A\in{\mathcal A}\right\}$. Recall
that for a sequence of open covers $\left({\mathcal G}_n\right)_n$
of $X$, then $St\left(x,{\mathcal
G}_n\right)=\bigcup\left\{G\in{\mathcal G}_n\mid x\in G\right\}$,
for every $x\in X$ and $n\in{\mathbb N}$.

A sequence of open covers $\left({\mathcal G}_n\right)_n$ of $X$
is called a
\begin{itemize}
\item \emph{development} and the space \emph{developable} if, for
every $x\in X$, the sequence $\left(st\left(x,{\mathcal
G}_n\right)\right)_n$ is a base at $x$. Furthermore, a regular
developable space is called \emph{a Moore space}. \item {weak
development} and the space \emph{weakly developable} if, for every
$x\in X$ and $\left(G_n\right)_n$ such that $x\in G_n\in{\mathcal
G}_n$ for every $n$, then the sequence $\left(\bigcap_{i\leq
n}G_i\right)_n$ is a base at $x$. \item \emph{weak
$k$-development} and the space \emph{weakly $k$-developable} if,
for every compact subset $K$ of $X$ and every ${\mathcal
H}_n\subset {\mathcal G}_n$ such that, for every $n$, ${\mathcal
H}_n$ is finite, $K\cap H\not=\emptyset$ for every $H\in{\mathcal
H}_n$, and $K\subset \bigcup{\mathcal H}_n$, then the sequence
$\left(\bigcap_{i\leq n}\left(\bigcup{\mathcal
H}_i\right)\right)_n$ is a base at $K$. \end{itemize}

The notions of developable space and weakly $k$-developable space
are different and both stronger than that of weakly developable
space. See \cite{aac}.

A sequence of open covers $\left({\mathcal G}_n\right)_n$ of a
space $X$ is called a
\begin{itemize}
\item \emph {$G_\delta$-diagonal sequence} if, for every $x\in X$,
$\bigcap_n St\left(x,{\mathcal G}_n\right)=\left\{x\right\}$.
\item \emph{$G_\delta^\ast$-diagonal sequence} if, for every $x\in
X$, $\bigcap_n{\overline{St\left(x,{\mathcal
G}_n\right)}}=\left\{x\right\}$. \item \emph{weakly
$G_\delta^\ast$-diagonal sequence} (see \cite{all1}) if, for every
$x$ and $\left(G_n\right)_n$ such that $x\in G_n\in{\mathcal G}_n$
for every $n$, then $\bigcap_n({\overline{\bigcap_{k\leq
n}G_k}})=\left\{x\right\}$.
\end{itemize}

Every $G_\delta^\ast$-diagonal sequence is a weakly
$G_\delta^\ast$-diagonal sequence, every weakly
$G_\delta^\ast$-diagonal sequence is a $G_\delta$-diagonal
sequence, and from every $G_\delta$-diagonal sequence on a regular
space, we can construct a weakly $G_\delta^\ast$-diagonal
sequence. The converse is false in general (see \cite{all1}).

It is well known (Ceder \cite{ceder}) that a space has a
$G_\delta$-diagonal sequence if and only if the diagonal
$\Delta=\left\{(x,x)\mid x\in X\right\}$ is a $G_\delta$-set in
$X\times X$. In this case, one says that the space \emph{has a
$G_\delta$-diagonal}.

A sequence of open covers $\left({\mathcal G}_n\right)_n$ of a
space $X$ is called a
\begin{itemize}
\item \emph{$w\Delta$-sequence} and the space a
\emph{$w\Delta$-space}, if for every $x\in X$ and
$\left(x_n\right)_n$ such that $x_n\in St\left(x,{\mathcal
G}_n\right)$ for every $n$, the sequence $\left(x_n\right)_n$ has
a cluster point. \item \emph{weakly $w\Delta$-sequence} and the
space a \emph{weakly $w\Delta$-space} (see \cite{all1}), if for
every $x\in X$ and $\left(G_n\right)_n$ such that $x\in
G_n\in{\mathcal G}_n$ for every $n$, then whenever
$x_n\in\bigcap_{k\leq n}G_n$ for every $n$, the sequence
$\left(x_n\right)_n$ has a cluster point.
\end{itemize}

It is well known (Hodel \cite{hodel}) that a space is developable
if and only if it is $w\Delta$-space and has a
$G_\delta^\ast$-diagonal. Also a space is weakly developable if
and only if it is weakly $w\Delta$-space and has a weakly
$G_\delta^\ast$-diagonal. See \cite{all1}.

The notion of $p$-space, first introduced by A. V.
Arhangel'ski\u{\i} in \cite{a8} for completely regular spaces, has
been internally characterized by D. K. Burke in \cite{bu2} as
follows: A space $X$ is a $p$-space if and only if it has a
sequence of open covers $\left({\mathcal G}_n\right)_n$ such that
whenever $x\in X$, $G_n\in{\mathcal G}_n$ and $x\in G_n$ for every
$n$, then $\bigcap_n{\overline{G_n}}$ is compact and every open
neighborhood of $\bigcap_n{\overline{G_n}}$ contains some
$\bigcap_{k\leq n}{\overline{G_k}}$. Such sequence will be called
a \emph{$p$-sequence}. The notions of $p$-space and
$w\Delta$-space are different and both strictly stronger than that
of weakly $w\Delta$-space (see \cite{all1}]. Every countably
compact space which is not a $k$-space provides us with an example
of a $w\Delta$-space which is not a $p$-space (see
\cite[Example~3.23]{ggr}). On the other hand, the Gruenhage's
space (see \cite[Example~2.17]{ggr} or \cite[Example~3]{aac}) is a
locally compact submetrizable space. This space is
weakly-$k$-developable and hence a $p$-space but it is not a
$w\Delta$-space. Let us point out here that regular weakly
developable spaces are in fact $p$-spaces.
\begin{prop}
A regular space is weakly developable if and only if it is a $p$-space
and has a $G_\delta$-diagonal.
\end{prop}
\begin{proof}
Let $X$ be a weakly developable space and let $({\mathcal G}_n)$
be its weak development. We may assume without loss of generality
that ${\mathcal G}_1=\{X\}$ and $\overline{\mathcal G}_{n+1}$
refines ${\mathcal G}_n$ for every $n$. It is easy to check that
$({\mathcal G}_n)$ is the required $G_{\delta}$-diagonal sequence
and it is a $p$-sequence where $\bigcap_n{\overline{G_n}}$ is
exactly the singleton $\{x\}$.

Conversely, we can construct a $G_{\delta}$-diagonal sequence
which is also a $p$-sequence. This sequence is then the required
weak development for $X$.
\end{proof}

The following result could be seen as a generalization of
\cite[Corollary~3.4]{ggr} and provides us with another
factorization of weakly developable spaces, developable spaces and
metrizable spaces.
\begin{thm}\label{stratif}
Let $X$ be a topological space.
\begin{enumerate}
\item $X$ is weakly developable if and only if it is a weakly
$w\Delta$-space and has a weakly $G_{\delta}^\ast$-diagonal. \item
$X$ is developable if and only if it is a submetacompact weakly
$w\Delta$-space and has a weakly $G_{\delta}^\ast$-diagonal. \item
$X$ is metrizable if and only if it is a paracompact weakly
$w\Delta$-space and has a weakly $G_{\delta}^\ast$-diagonal.
\end{enumerate}
\end{thm}
\begin{proof}
The first statement is a result of \cite{all1}. The second
statement follows from the first and the fact that every
submetacompact weakly developable space is developable (see
\cite{aac}). The third statement follows from the second and the
fact that every paracompact developable space is metrizable.
\end{proof}
Combining the above theorem with the fact that semi-stratifiable
spaces are metacompact (see \cite[Theorem~2.6]{css}) and have a
$G_{\delta}$-diagonal, and with the fact that stratifiable spaces
are paracompact, we obtain the following result which can be seen
as a generalization of \cite[Corollary~5.12]{ggr}.
\begin{cor}\label{stratif2}
Let $X$ be a regular topological space.
\begin{enumerate}
\item $X$ is developable if and only if it is a semi-stratifiable
weakly $w\Delta$-space. \item $X$ is metrizable if and only if it
is a stratifiable weakly $w\Delta$-space.
\end{enumerate}
\end{cor}
\section{Weakly developable spaces versus semi-metrizable spaces}
In this section, we investigate weakly developable spaces and
semi-metrizable spaces.

Let $X$ be a nonempty set and $d: X \times X \longrightarrow \R^+$
be a function. Recall tat the function $d$ is called a distance
function on the set $X$ if
\begin{description}
\item [(W1)] $d\left(x,y\right)=0\Longleftrightarrow x=y$; \item
[(W2)]$d\left(x,y\right)=d\left(y,x\right)$, for every $x,y\in X$.
\end{description}
One usually denotes the ball around $x\in X$ with radius
$\varepsilon>0$ by $$\displaystyle
B\left(x,\varepsilon\right)=\left\{y\in X\mid
d\left(x,y\right)<\varepsilon\right\}$$ and, for $x\in X$ and
$A\subset X$, one puts $$\displaystyle d\left(x,
A\right)=\inf\left\{d\left(x,y\right)\mid y\in A\right\}.$$ A
distance function $d$ on $X$ induces a topology ${\mathcal T}_d$
on $X$ in the following way: $$U\in{\mathcal
T}_d\Longleftrightarrow \forall x\in U, \exists\varepsilon>0,
B\left(x,\varepsilon\right)\subset U.$$ This topology is always
$T_1$. A topological space $X$ is called
\begin{itemize}
\item \emph{symmetrizable} if its topology is equal to ${\mathcal
T}_d$ where $d$ is a distance function on $X$. \item
\emph{semi-metrizable} if it is symmetrizable and
$B\left(x,\varepsilon\right)$ is a neighborhood of $x$, for every
$x\in X$ and $\varepsilon>0$.
\end{itemize}

\begin{rem}
Semi-metrizable spaces are semi-stratifiable and then
submetacompact. Thus, by Theorem~\ref{stratif}, the notions of
weakly developable space and semi-metrizable space are different.
Every weakly developable non developable space is not
semi-metrizable and every semi-metrizable non developable space is
not weakly developable.
\end{rem}

Now, we introduce and study a notion which generalizes both weakly
developable spaces and semi-metrizable spaces.

Let $X$ be a nonempty set and let $d$ be a distance function on
$X$. We say that a family ${\mathcal V}=\left\{V_n^x\mid x\in X,
n\in\N\right\}$ of subsets of $X$ is \emph{adapted} to $d$ if
\begin{displaymath}x\in\bigcap_{k=1}^nV_k^x\subset
B\left(x,2^{-n}\right),\forall x\in X, \forall
n\in\N.\end{displaymath} In this case, we introduce a topology
${\mathcal T}_{d,{\mathcal V}}$ on $X$ in the following way:
\begin{displaymath}U\in{\mathcal T}_{d,{\mathcal V}}\Longleftrightarrow
\forall x\in U, \exists n\in\N, \bigcap_{k=1}^n V_k^x\subset
U.\end{displaymath} The space $\left(X,{\mathcal T}_{d,{\mathcal
V}}\right)$ is $T_1$-topological space and ${\mathcal
T}_{d}\subset{\mathcal T}_{d,{\mathcal V}}$. In the sequel, the
notation $\left(X,d,{\mathcal V}\right)$ stands for the
\emph{weakly symmetric space} $X$. A topological space $X$ will be
said a \emph{weakly symmetrizable space} if its topology is equal
to ${\mathcal T}_{d,{\mathcal V}}$ for some distance function $d$
on $X$ and an adapted family ${\mathcal V}=\left\{V_n^x\mid x\in
X, n\in\N\right\}$ of subsets of $X$ to $d$.

Obviously, every symmetrizable space is weakly symmetrizable. The
converse does not hold in general (see Remark~\ref{rwdsemi}
below). A first examination of weakly symmetrizable spaces yields
the following property inherited from symmetrizable spaces.
\begin{prop} Every weakly symmetric space $\left(X,d,{\mathcal V}\right)$ is
sequential.
\end{prop}
\begin{proof} Let $A$ be a non closed subset of $X$. Then, there exists $x\notin A$ such that
$\bigcap_{k=1}^n V_k^x\cap A\not=\emptyset$, for every $n\in\N$.
Pick $x_n\in\bigcap_{k=1}^n V_k^x\cap A$ and claim that
$\left(x_n\right)_n$ converges to $x$. Let $U$ be an open
neighborhood of $x$. Let $m_0\in\N$ be such that
$\bigcap_{k=1}^{m_0} V_k^x\subset U$. It follows that $x_n\in U$
for every $n\geq m_0$ and thus $\left(x_n\right)_n$ converges to
$x$. Consequently the subset $A$ is not sequentially closed.
\end{proof}

Let $\left(X,d,{\mathcal V}\right)$ be a weakly symmetric space. A
sequence $\left(x_n\right)_n$ in $X$ will be said \emph{${\mathcal
V}$-converging} to $x\in X$ and we write
\emph{$x_n\stackrel{{\mathcal V}}{\longrightarrow}x$} if
\begin{displaymath}\left\{x_n\mid n\in \N\right\}\setminus
\cap_{k=1}^m V_k^x \textrm{ is finite, for each }
m\in\N.\end{displaymath} A sequence $\left(x_n\right)_n$ of
elements of $X$ is said to be \emph{injective} if the set
$\left\{x_n\mid n\in N\right\}$ is infinite.

The following result is easy to prove.
\begin{lem}\label{caract1}
Let $\left(X,d,{\mathcal V}\right)$ be a weakly symmetric space.
If $\left(x_n\right)_n$ is an injective sequence in $X$ such that
$x_n\stackrel{{\mathcal V}}{\longrightarrow}x$, then
$x_n\longrightarrow x$ and
$\lim\limits_{n\to+\infty}d\left(x_n,x\right)=0$.
\end{lem}

The following lemma further clarifies the relation of the adapted
family ${\mathcal V}$ and the distance function $d$ to the
topology of a weakly symmetrizable space.

\begin{lem}\label{caract2}
Let $\left(X,d,{\mathcal V}\right)$ be a weakly symmetric space.
We have the following statements:
\begin{enumerate}\item If every subset consisting of a
converging sequence and its limit is closed (in particular, if $X$
is Hausdorff), then for every sequence $\left(x_n\right)_n$ in
$X$, $x_n\longrightarrow x$ implies $x_n\stackrel{{\mathcal
V}}{\longrightarrow}x$. \item ${\mathcal T}_{d}={\mathcal
T}_{d,{\mathcal V}}$ if and only if, for every sequence
$\left(x_n\right)_n$ in $X$,
$\lim\limits_{n\to+\infty}d\left(x_n,x\right)=0$ implies
$x_n{\longrightarrow}x$.
\end{enumerate}
\end{lem}
\begin{proof}
The second statement being well known, we prove only the first
statement. Let $\left(x_n\right)_n$ be a converging sequence to
$x$ and suppose
\begin{displaymath}
A=\left\{x_n\mid n\in \N\right\}\setminus \bigcap_{k=1}^{m_0}
V_k^x \end{displaymath}
is infinite for some $m_0\in\N$. The set
$A$ is a converging sub-sequence of $\left(x_n\right)_n$ to
$x\notin A$ and then $A$ is not closed. On the other hand, we will
prove that $X\setminus A$ is open which is impossible. Let $y\in
X\setminus A$. If $y=x$, we have
$x\in\bigcap_{k=1}^{m_0}V_k^x\subset X\setminus A$. If $y\not=x$,
then $y\in X\setminus A\cup\left\{x\right\}$. Since
$A\cup\left\{x\right\}$ is closed, let $m_y\in\N$ be such that
$y\in\bigcap_{k=1}^{m_y} V_k^x\subset X\setminus
A\cup\left\{x\right\}\subset X\setminus A$.
\end{proof}

\begin{thm}
Let $\left(X,d,{\mathcal V}\right)$ be a weakly symmetric space
and consider the following conditions:
\begin{enumerate}
\item For every $x\in X$, the family $\left\{\bigcap_{k=1}^n
V_k^x\mid n\in\N\right\}$ is a local base at $x$; \item $X$ is
first countable space; \item $X$ is Fr\'echet space.
\end{enumerate}
Then $(1)\Longrightarrow (2)\Longrightarrow (3)$. If in addition,
for every injective sequence $\left(x_n\right)_n$ in $X$,
$x_n\longrightarrow x$ if and only if $x_n\overset{{\mathcal
V}}{\longrightarrow}x$, then $(3)\Longrightarrow (1)$.
\end{thm}
\begin{proof}
Assume that $X$ is a Fr\'echet space. Suppose that for some $x\in
X$ and $m_0\in\N$, the set $\bigcap_{k=1}^{m_0} V_k^x$ is not a
neighborhood of $x$. For every neighborhood $V$ of $x$, choose
$x_V\in V\setminus\bigcap_{k=1}^{m_0} V_k^x$ and put
$A=\left\{x_V\mid V\textrm{ is a neighborhood of } x\right\}$.
Since $X$ is Fr\'echet, there exists an injective sequence
$\left(x_n\right)_n$ in $A$ converging to $x$. It follows that
$\left(x_n\right)_n$ is ${\mathcal V}$-converging to $x$ and then,
$\bigcap_{k=1}^{m_0} V_k^x$ contains infinitely many elements of
$A$. Contradiction.
\end{proof}
For a subset $A$ of $X$, we define the ${\mathcal V}$-closure
$C_{{\mathcal V}}\left(A\right)$ of $A$ as follows:
\begin{displaymath}C_{{\mathcal V}}\left(A\right)=A\cup\left\{x\in X\mid
\exists\left(x_n\right)_n\subset A, \left(x_n\right)_n \textrm{
injective and } x_n\overset{{\mathcal
V}}{\longrightarrow}x\right\}.\end{displaymath} Clearly,
$C_{{\mathcal V}}\left(A\right)\subset\overline{A}$ where
$\overline{A}$ is the closure of $A$ with respect to the topology
${\mathcal T}_{d,{\mathcal V}}$. The converse holds under
additional conditions.
\begin{thm}
Let $\left(X,d,{\mathcal V}\right)$ be a weakly symmetric space.
The following two conditions are equivalents:
\begin{enumerate}
\item For every $A\subset X$, we have $\overline{A}=C_{{\mathcal
V}}\left(A\right)$, \item $X$ is Fr\'echet space, and for every
injective sequence $\left(x_n\right)_n$ in $X$,
$x_n\longrightarrow x$ if and only if $x_n\overset{{\mathcal
V}}{\longrightarrow}x$. \end{enumerate}
\end{thm}
\begin{proof}
$(1)\Longrightarrow (2)$: To prove that $X$ is Fr\'echet, it
suffices to prove that for every $\in X$, the family
$\left\{\bigcap_{k=1}^n V_k^x\mid n\in\N\right\}$ is a local base
at $x$. As in the above theorem, suppose that for some $x\in X$
and $m_0\in\N$, the set $\bigcap_{k=1}^{m_0} V_k^x$ is not a
neighborhood of $x$. Then, $x\in \overline{A}$ where $A=X\setminus
\bigcap_{k=1}^{m_0} V_k^x$. Since $\overline{A}=C_{{\mathcal
V}}\left(A\right)$ and $x\notin A$, there exists an injective
sequence $\left(x_n\right)_n$ in $A$ such that
$x_n\overset{{\mathcal V}}{\longrightarrow}x$. Thus,
$\bigcap_{k=1}^{m_0} V_k^x$ contains an infinitely many elements
of $A$. Contradiction. To prove the second assertion, let
$\left(x_n\right)_n$ be an injective sequence in $X$ such that
$x_n\longrightarrow x$. Without loss of generality, we may assume
that $x_n\not=x$, for every $n\in\N$. If $\left(x_n\right)_n$ is
not ${\mathcal V}$-converging to $x$, then, for some $m_0\in\N$,
the set
\begin{displaymath}\left\{x_n\mid n\in \N\right\}\setminus
\cap_{k=1}^{m_0} V_k^x\end{displaymath} is infinite. Thus, there
is a subsequence $\left(x_{n_k}\right)_k$ such that
$x_{n_k}\notin\cap_{k=1}^{m_0} V_k^x$, for every $k\in\N$. Put
$A=\left\{x_{n_k}\mid k\in \N\right\}$. Clearly $x\notin A$ and
$x\in\overline{A}=C_{{\mathcal V}}\left(A\right)$. Then, there
exists an injective sequence $\left(y_n\right)_n$ in $A$ which is
${\mathcal V}$-converging to $x$. Thus, $\cap_{k=1}^{m_0} V_k^x$
contains an infinitely many elements of $A$. Contradiction.

\noindent $(2)\Longrightarrow (1)$: Let $A\subset X$ and $x\in
\overline{A}$. If $x\in A$, then $x\in C_{{\mathcal
V}}\left(A\right)$. Suppose $x\notin A$ and since $X$ is a
Fr\'echet space, there exists a sequence $\left(x_n\right)_n$ in
$A$ converging to $x$. Without loss of generality, we may assume
that $\left(x_n\right)_n$ is injective. Then, by assumption, the
sequence $\left(x_n\right)_n$ is ${\mathcal V}$-converging to $x$
which completes the proof.
\end{proof}

A weakly symmetric space $\left(X,d,{\mathcal V}\right)$ will be
said \emph{weakly semi-metric} if in addition
$\overline{A}=C_{{\mathcal V}}\left(A\right)$ for every $A\subset
X$, and the following condition holds:
\begin{description}
\item[(WS-M)] if $x\in\bigcap_n V_n^{x_n}$, then $\bigcap_n
V_n^{x_n}=\left\{x\right\}$ and whenever $x_n\in
\bigcap_{k=0}^nV_k^{x_k}$ for every $n$, the sequence
$\left(x_n\right)_n$ converges to $x$.
\end{description}
A space $X$ will be called a \emph{weakly semi-metrizable space}
if its topology is equal to ${\mathcal T}_{d,{\mathcal V}}$ of a
weakly semi-metric space $\left(X,d,{\mathcal V}\right)$. From the
definition, we have the following result:
\begin{prop} Every weakly semi-metrisable space is first countable and has a $G_{\delta}$-diagonal.
\end{prop}
Like for developable spaces (see \cite{lcds}), the following
result gives some properties of the distance function defined on a
weakly developable space.
\begin{thm}\label{wdsemi} Every weakly developable space is
weakly semi-metrizable. Moreover, the compatible distance function
$d$ and its adapted family $\left\{V_n^x\mid x\in X,
n\in\N\right\}$ satisfy the following additional conditions: for
every $x\in X$ and $n\in\N$, the set $V_n^x$ is open and
$d\left(y_1,y_2\right)<2^{-n}$, for every $y_1$ and $y_2$ in
$V_{n}^x$.
\end{thm}
\begin{proof}
Let $X$ be a weakly developable space and let $\left({\mathcal
G}_n\right)$ be its weak development. For every $x\in X$ and $y\in
X$, put
\begin{displaymath}d\left(x,y\right)=\inf\left\{2^{-n}\mid
\exists G_n\in{\mathcal G}_n, x\in G_n{\textrm{ and }} y\in
G_n\right\}.\end{displaymath} It is clear that $d$ is a distance
function on $X$. On the other hand, for every $x\in X$ and
$n\in\N$, fix $V_n^x\in{\mathcal G}_n$ such that $x\in V_n^x$. It
is also clear that $\left\{V_n^x\mid x\in X, n\in\N\right\}$ is
the required family.
\end{proof}
Recall the following axiom introduced by Wilson for symmetric
spaces. Let $\left(X,d\right)$ be a symmetric space. The space
$\left(X,d\right)$ is said to be satisfying (W3) if
\begin{description}
\item[(W3)]whenever $\left(x_n\right)_n$ is a sequence in $X$ and
$x,y\in X$, we have
$\lim\limits_{n\to+\infty}d\left(x_n,x\right)=\lim\limits_{n\to+\infty}d\left(x_n,y\right)=0$
implies $x=y$.
\end{description}
Every semi-metrizable Hausdorff space satisfies (W3). Moreover, we
have the following.
\begin{thm}\label{semisemi}
Every semi-metrizable space satisfying (W3) is weakly
semi-metrizable.
\end{thm}
By using Theorem~\ref{wdsemi} and Theorem~\ref{semisemi}, it
follows that the notion of a weakly semi-metrizable space
generalizes both the notions of a weakly developable space and a
semi-metrizable space.
\begin{rem}\label{rwdsemi} By applying Theorem~\ref{stratif}, every
weakly developable non developable space provides us with an
example of a weakly semi-metrizable space which is not
symmetrizable (and then, not semi-metrizable). On the other hand,
every semi-metrizable non developable Hausdorff space provides us
with an example of a weakly semi-metrizable space which is not
weakly developable.
\end{rem}

The following diagram summarizes some notions of generalized
metrizable spaces used in the paper and close to the notion of a weakly semi-metrizable space.\\

\noindent {\tiny
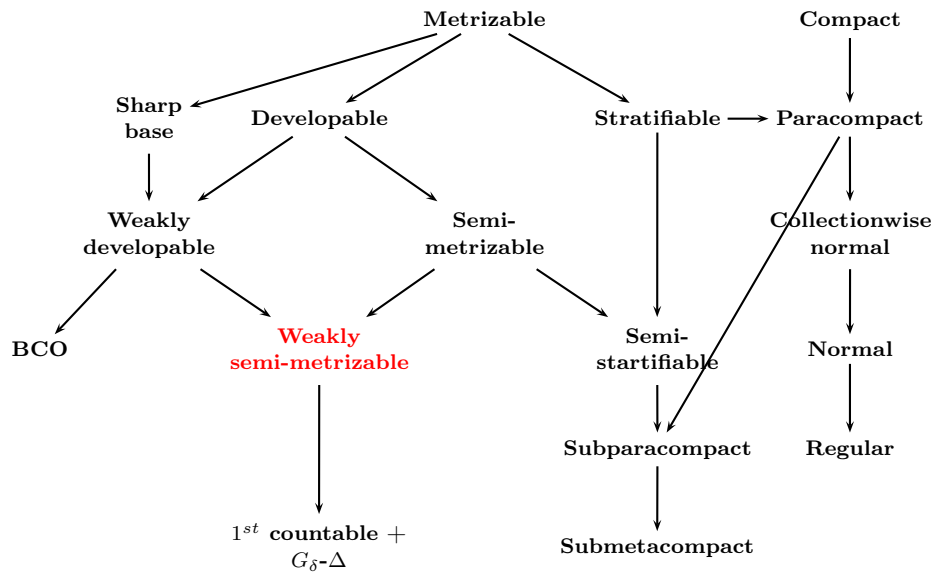
{
\begin{psmatrix}[rowsep=0.8cm,colsep=0.2cm]
&&&{\textbf{Metrizable}}&&\textbf{Compact}\\
&\deuxlignes{\textbf{Sharp}\\\textbf{base}}&\textbf{Developable}&&\textbf{Stratifiable}&\textbf{Paracompact}\\
&\deuxlignes{\textbf{Weakly}\\\textbf{developable}}&&\deuxlignes{\textbf{Semi-}\\\textbf{metrizable}}&&\deuxlignes{\textbf{Collectionwise}\\
\textbf{normal}}\\
\textbf{BCO}&&\deuxlignes{\red{\textbf{Weakly}}\\\red{\textbf{semi-metrizable}}}
&&\deuxlignes{\textbf{Semi-}\\\textbf{startifiable}}&\textbf{Normal}\\
&&&&\textbf{Subparacompact}&\textbf{Regular}\\
&&\deuxlignes{\textbf{$1^{st}$ countable
+}\\\textbf{$G_\delta$-$\Delta$}}&&\textbf{Submetacompact}&
\psset{arrows=->,nodesep=1mm,labelsep=1.5mm}
\everypsbox{\scriptstyle} \ncline{1,6}{2,6} \ncline{1,4}{2,2}
\ncline{1,4}{2,3} \ncline{1,4}{2,5} \ncline{2,2}{3,2}
\ncline{2,6}{3,6} \ncline{2,3}{3,2} \ncline{2,3}{3,4}
\ncline{2,5}{4,4} \ncline{2,6}{5,5} \ncline{2,5}{2,6}
\ncline{2,5}{4,5} \ncline{3,2}{4,1} \ncline{3,2}{4,3}
\ncline{3,4}{4,5} \ncline{3,4}{4,3} \ncline{3,4}{4,4}
\ncline{3,6}{4,6} \ncline{4,5}{5,5} \ncline{4,3}{6,3}
\ncline{4,6}{5,6} \ncline{5,5}{6,5}
\end{psmatrix}
}}

\section{Extension of continuous functions: comments and remarks}
Historically, semi-metrizable spaces and developable spaces have
been introduced independently by Fr\'echet and Moore respectively.
The first as a generalization of metric spaces, the second to deal
with a form of general analysis. While generalized metrizable
spaces have been intensively studied by several authors since
their introduction, the modern era of semi-metrizable spaces began
with the important work of Jones, Heath and McAuley. The
metrization theorem of the factorization type of Bing is one of
the most influential results of a long and fruitful period of
research in metrization theory and the theory of generalized
metrizable spaces. Although, the notion of a weakly developable
space is introduced more recently, it has been considered by
several authors in the last decade. It fits very well in the class
of generalized metric spaces and allows to obtain various new
results. See for example,
\cite{all1,aac,ac2,s6,s4,balbu,s5,gut,lub1,lub2,leimou,s7}.

\subsection{Tietze's extension theorem and its generalizations}
There is a long history between generalized metrizable spaces and
the problem of extending continuous functions. Generalized metric
notions are often characterized by means of results on extension
of continuous functions. The Tietze-Urysohn extension theorem is
one of the most important results in this context. Several other
deep results on extension of continuous functions exist in the
literature and in particular, Dugundji extension theorem for
metric spaces and Borges extension theorem for stratifiable
spaces.

Another form of the problem of extending continuous functions is
the concept of absolute (neighborhood) extensor and absolute
(neighborhood) retract. This approach has several applications to
different areas of mathematics including applications to
topological methods for partial differential equations, ordinary
differential equations and differential inclusions. As well known,
this concept has been considered by several authors and in
particular, by Borsuk, Kuratowski, Fox, Dowker, Dugundji, Hu,
Banner, Michael and by many other mathematicians. For further
details on infinite-dimensional topology, we refer to
\cite{bess,mill,nad}. In this context and especially when dealing
with the problem of extending set-valued functions (see
Curtis-Schori theorem, for example), various concepts of
completeness such as Cauchy completeness and McAuley completeness
as well as notions of connectedness like arc-wise connectedness,
path-wise connectedness, connectedness and local connectedness are
often considered in the settings of metrizable and developable
spaces. Note that many results in this direction have been
obtained recently in the settings of stratifiable spaces and their
generalizations by T. Banakh who introduced the notion of a
quarter-stratifiable space (see \cite{banak}).

As mentioned in \cite{heath}, many theorems concerning developable
spaces have analogies in semi-metric spaces. This suggests the
question of whether this assertion remains true for weakly
semi-metrizable spaces. It turns out that many results related to
the problem of extending continuous functions have already been
considered on semi-metrizable spaces but not yet on weakly
developable spaces.

\subsection{E. Michael selection theory}
An important point of view of the problem of extending continuous
functions is Michael's selection theory. It treats Tietze-Urysohn
extension as a special case and gives characterizations for some
generalized metrizable spaces such as normal spaces,
collectionwise normal spaces, fully normal spaces and paracompact
spaces. For a complete presentation of the work of E. Michael, we
refer to \cite{rep2}, and for an overview of recent results on the
theory of continuous selections of set-valued mappings, we refer
to \cite{rep1}.

Let $f:A\subset X\longrightarrow Y$ be a continuous function. One
defines a lower semicontinuous set-valued function $F$ on $X$ by
\begin{eqnarray*}
F\left(x\right)=
\begin{cases}\left\{f(x)\right\} &\text{if}\quad x\in A,\\
Y &\text{otherwise.}
\end{cases}
\end{eqnarray*}
Clearly, any continuous selection of $F$ is a continuous extension
of $f$.

In order to generalize a result of E. Michael on double selection,
J. Calbrix and the present author introduced the notion of weakly
developable spaces. Some results have been first obtained in
\cite{ca} for \u{C}ech-complete weakly $k$-developable spaces and
generalized later in \cite{ac2} for spaces having a monotonically
complete base of countable order (BCO). New results on selection
theory of set-valued mappings have been also obtained recently by
several authors (see for example, \cite{gut,s3,s1}). As showed in
\cite{s2}, note that having a BCO in the selection theorem
obtained in \cite{s1} for set-valued mappings defined on
zero-dimensional paracompact spaces with values in spaces having a
BCO can not be replaced by being stratifiable.

For convenience, recall that a \u{C}ech-complete space is a
completely regular space which is a $G_\delta$-set in some
Hausdorff compact space. A metrizable space is completely
metrizable if and only if it is \u{C}ech-complete. An internal
characterization of \u{C}ech-complete spaces due to
Arhangel'ski\u{\i} \cite{ssmt} and Frolik \cite{goftg} by means of
sequence of open covers is the following: A Hausdorff space $X$ is
said to be A.F.-complete if it has an A.F.-complete sequence
$\left({\mathcal G}_n\right)$ of open covers. That is, for every
centered family ${\mathcal F}$ of subsets of $X$, if
$$\delta\left({\mathcal F}\right)<{\mathcal G}_n, \text{ for every } n,
\text{ then } \bigcap\left\{F\mid F\in{\mathcal
F}\right\}\not=\emptyset$$ where $\delta\left({\mathcal
F}\right)<{\mathcal G}_n$ means that there exists $F\in {\mathcal
F}$ and $G\in {\mathcal G}_n$ such that $F\subset G$. A completely
regular space is A.F.-complete if and only if it is
\u{C}ech-complete.

A base ${\mathcal B}$ for a space $X$ is called a base of
countable order (BCO) if for every $x\in X$ and every strictly
decreasing sequence $\left(G_n\right)_n$ of elements of ${\mathcal
B}$ containing $x$, $\left(G_n\right)_n$ is a base at $x$. Every
weakly developable space has a BCO. A space is metrizable if and
only if it is paracompact and has a BCO (see \cite{ssmt}) and a
space is developable if and only if it is submetacompact and has a
BCO (see \cite{ww3}). Thus, the notions of semi-metrizable space
and space having a BCO are different.

Like A.F.-complete spaces, spaces having a monotonically complete
BCO is an other generalization of complete metric spaces. A base
${\mathcal B}$ is said to be monotonically complete base if for
every decreasing sequence $\left(G_n\right)_n$ of elements of
${\mathcal B}$, we have $\bigcap_n{\overline B_n}\not=\emptyset$
(see \cite{ww1,ww3,ww2}).

Roughly speaking, proofs of the results on existence of selection
of set-valued mappings with values on a space $Y$ having a BCO are
generally based on a construction of an open continuous mapping
$g$ from a suitable metrizable space $M$ onto $Y$ such that
$g^{-1}\left(y\right)$ is complete with respect to a fixed
distance $d$ on $M$, for every $y\in Y$. This is possible thanks
to monotonic completeness of the base of countable order.

One might then ask what kind of completeness we can define on
weakly semi-metrizable spaces to obtain similar results. Thus, it
is not clear whether these results on selection of set-valued
mappings which are obtained on weakly developable spaces, and then
on developable spaces, remain also true on semi-metrizable spaces.

\subsection{Fixed point theory}
Fixed point theory which is a large and an important field of
mathematics has to be considered here too. Although, there are
many results on fixed point theory in the settings of metrizable
spaces, there is recently a real interest in generalization of
these results to semi-metrizable spaces. (See for example,
\cite{ssafp} and the references therein). Several adaptations of
the notion of a contraction mapping on metric spaces have been
already given for semi-metrizable spaces and for other concepts of
generalized metrisable spaces. See for example \cite{chob}, for
multi-metric spaces.

Weakly semi-metrizable spaces provide us with a distance function
which can be useful to obtain similar adaptations. It should be
interesting to know what kind of fixed point theorems one can
obtain on weakly developable spaces and more generally, on weakly
semi-metrizable spaces.

\bibliographystyle{plain}
%\bibliography{ecfwds}

\begin{thebibliography}{10}

\bibitem{all1}
B.~{Alleche}.
\newblock {Weakly developable and weakly $k$-developable spaces, and the
  {V}ietoris topology}.
\newblock {\em Topology Appl.}, 111:3--19, 2001.

\bibitem{aac}
B.~{Alleche}, A.~V. {Arhangel'ski\u{\i}}, and J.~{Calbrix}.
\newblock {Weak developments and metrization}.
\newblock {\em Topology Appl.}, 100(1):23--38, 2000.

\bibitem{ac2}
B.~{Alleche} and J.~{Calbrix}.
\newblock {Semi-continuous multifunctions and bases of countable order}.
\newblock {\em Topology Appl.}, 104:3--12, 2000.

\bibitem{ssafp}
I.~D. {Aran{\dj}elovi\'c} and D.~J. {Ke\v{c}ki\'c}.
\newblock {Symmetric spaces approach to some fixed point results}.
\newblock {\em Nonlinear Analysis}, 75:5157--5168, 2012.

\bibitem{a8}
A.~V. {Arhangel'ski\u{\i}}.
\newblock {On a class contening all metric and all locally bicompact spaces}.
\newblock {\em Soviet Math. Dokl.}, 4:1051--1055, 1963.

\bibitem{ssmt}
A.~V. {Arhangel'ski\u{\i}}.
\newblock {Some metrization theorems}.
\newblock {\em Uspekhi Mat. Nauk}, 18:139--145, 1963.
\newblock in Russian.

\bibitem{s6}
A.~V. {Arhangel'ski\u{\i}}, W.~Just, E.~A. Rezniczenko, and P.~J.
Szeptycki.
\newblock {Sharp bases and weakly uniform bases versus point-countable bases}.
\newblock {\em Topology Appl.}, 100:39--46, 2000.

\bibitem{s4}
B.~Bailey and G.~Gruenhage.
\newblock {On a question concerning sharp bases}.
\newblock {\em Topology Appl.}, 153:90--96, 2005.

\bibitem{balbu}
Z.~Balogh and D.~K. Burke.
\newblock {Two results on spaces with a sharp base}.
\newblock {\em Topology Appl.}, 154:1281--1285, 2007.

\bibitem{banak}
T.~Banakh.
\newblock {(Metrically) quarter-stratifiable spaces and their applications in
  the theory of separately continuous functions}.
\newblock {\em Matematychni Studii}, 18:10--28, 2002.

\bibitem{bess}
C.~Bessaga and A.~Pelczytfski.
\newblock {\em {Selected topics in infinite-dimensional topology}}.
\newblock PWN-Polish Scientific Publishers, Warszawa, 1975.

\bibitem{bu2}
D.~K. {Burke}.
\newblock {On $p$-spaces and $w{\Delta}$-spaces}.
\newblock {\em Pacific J. Math.}, 11:105--126, 1970.

\bibitem{bu}
D.~K. {Burke}.
\newblock {Covering properties}.
\newblock In K.~{Kunen} and J.~E. {Vaughan}, editors, {\em {H}andbook of
  {S}et-theoretic {T}opology}, pages 347--422. Elsevier Science publishers,
  1984.

\bibitem{ca}
J.~{Calbrix} and B.~{Alleche}.
\newblock {Multifunctions and {\v{C}}ech-complete spaces}.
\newblock In P.~Simon, editor, {\em Proc. of the 8th Prague Topological
  Symposium, Prague, Czech Republic, August 18-24, 1996}, pages 30--36.
  Topology Atlas, 1997.

\bibitem{ceder}
J.~G. {Ceder}.
\newblock {Some generalisations of metric spaces}.
\newblock {\em Pacific J. Math.}, 11:105--125, 1961.

\bibitem{chob}
M.~M. Choban and L.~I. Calmu\c{t}chi.
\newblock {Fixed points theorems in multi-metric spaces}.
\newblock {\em Annals of the Academy of Romanian Scientists}, 3(1):46--68,
  2011.

\bibitem{css}
G.~{Creede}.
\newblock {Concerning semi-stratifiable spaces}.
\newblock {\em Pacific J. Math}, 32:47--54, 1970.

\bibitem{en}
R.~Engelking.
\newblock {\em {General {T}opology}}.
\newblock Heldermann, {V}erlag {B}erlin, 1989.

\bibitem{goftg}
Z.~{Frolík}.
\newblock {Generalization of the $G_\delta$ property of complete metric
  spaces}.
\newblock {\em Czechoslovak Math. J.}, 10:359--379, 1961.

\bibitem{s5}
C.~Good, R.~W. Knight, and A.~M. Mohamad.
\newblock {On the metrizability of spaces with a sharp base}.
\newblock {\em Topology Appl.}, 125:543--552, 2002.

\bibitem{ggr}
G.~Gruenhage.
\newblock {Generalized metric spaces}.
\newblock In K.~{Kunen} and J.~E. {Vaughan}, editors, {\em {H}andbook of
  {S}et-{T}heoretic {T}opology}, pages 423--501. {E}lsevier {S}cience
  {P}ublishers, 1984.

\bibitem{gut}
V.~Gutev.
\newblock {Completeness, Sections and Selections}.
\newblock {\em Set-Valued Anal.}, 15:275--295, 2007.

\bibitem{heath}
R.~W. Heath.
\newblock {Arc-wise connectedness in semi-metric spaces}.
\newblock {\em Pacific J. Math.}, 12:1301--1319, 1962.

\bibitem{hodel}
R.~{Hodel}.
\newblock {Moore spaces and $w{\Delta}$-spaces}.
\newblock {\em Pacific J. Math.}, 38:641--652, 1971.

\bibitem{lub1}
L.~Hol\'{a} and Z.~Piotrowski.
\newblock {Set of continuity points of functions with values in generalized
  metric spaces}.
\newblock {\em Tatra Mt. Math. Publ.}, 42:149--160, 2009.

\bibitem{lub2}
L.~Hol\'{a} and L.~Zsilinszky.
\newblock {Vietoris topology on partial maps with compact domains}.
\newblock {\em Topology Appl.}, 157(8):1439--1447, 2010.

\bibitem{lcds}
H.~Martin.
\newblock {Local connectedness in developable}.
\newblock {\em Pacific Journal of Mathematics}, 61(1):219--224, 1975.

\bibitem{leimou}
L.~{Mou} and H.~{Ohta}.
\newblock {Sharp bases and mappings}.
\newblock {\em Houston J. Math}, 31:227--238, 2005.

\bibitem{s7}
S.~A. Peregudov.
\newblock {Modifications of uniform bases and classification of topological
  spaces}.
\newblock {\em Journal of Mathematical Sciences}, 144(3):4184--4204, 2007.

\bibitem{rep1}
D.~{Repov\u{s}} and P.~V. Semenov.
\newblock {Continuous Selections of Multivalued Mappings}.
\newblock In M.~Hu\u{s}ek and J.~van Mill, editors, {\em Recent Progress in
  General {T}opology II}, pages 423--461. Elsevier Science B.V., 2002.

\bibitem{rep2}
D.~{Repov\u{s}} and P.~V. Semenov.
\newblock {Ernest Michael and theory of continuous selections}.
\newblock {\em Topology Appl.}, 155:755--763, 2008.

\bibitem{s3}
D.~{Repov\u{s}}, B.~{Tsaban}, and L.~{Zdomskyy}.
\newblock {Continuous selections and $\sigma$-spaces}.
\newblock {\em Topology Appl.}, 156:104--109, 2008.

\bibitem{mill}
J.~{van Mill}.
\newblock {\em {The Infinite-Dimensional Topology of Function spaces}}.
\newblock Elsevier, Amsterdam. The Netherlands, 2001.

\bibitem{nad}
A.~Wanes and S.~B. {Nadler, Jr}.
\newblock {\em {HYPERSPACES: Fundamentals and Recent Advances}}.
\newblock Marcel Dekker, Inc., New York, 1999.

\bibitem{ww1}
H.~H. {Wicke}.
\newblock {The regular open continuous images of complete metric spaces}.
\newblock {\em Pacific J. Math.}, 23:621--625, 1967.

\bibitem{ww3}
H.~H. {Wicke} and J.~M. {Worrell, Jr.}
\newblock {Characterizations of developable topological spaces}.
\newblock {\em Canad. J. Math.}, 17:820--830, 1965.

\bibitem{ww2}
H.~H. {Wicke} and J.~M. {Worrell, Jr.}
\newblock {Open continuous mappings of spaces having bases of countable order}.
\newblock {\em Duke Math. J.}, 34:255--271, 1967.

\bibitem{s1}
P.-F. {Yan} and S.-L. {Jiang}.
\newblock {Countable sets, BCO spaces and selections}.
\newblock {\em Topology Appl.}, 148:1--5, 2005.

\bibitem{s2}
P.-F. {Yan}, S.~{Lin}, and H.~{Yang}.
\newblock {On Finite Subsequence-covering Maps}.
\newblock {\em Advances In Mathematics}, 36:651--656, 2010.

\end{thebibliography}

\end{document}